\documentclass[11pt,reqno]{amsart}
\usepackage{fullpage}
\usepackage{amsmath,amsthm,verbatim,amssymb,amsfonts,amscd,graphicx,bbm}
\usepackage{graphics}
\usepackage{mathtools}
\usepackage{tikz-cd}
\usepackage{mathrsfs}
\usepackage{esint}
\usepackage{color}
\usepackage{hyperref}

\newtheorem{theorem}{Theorem}
\newtheorem{proposition}[theorem]{Proposition}
\newtheorem{lemma}[theorem]{Lemma}
\newtheorem{corollary}[theorem]{Corollary}

\theoremstyle{definition}
\newtheorem{remark}{Remark}[section]

\newcommand{\cref}[1]{Corollary~\ref{c.#1}}

\numberwithin{equation}{section}
\numberwithin{theorem}{section}

\newcommand{\Z}{\mathbb{Z}}

\newcommand{\R}{\mathbb{R}}

\newcommand{\bT}{\mathbb{T}}

\renewcommand{\tilde}{\widetilde}
\renewcommand{\div}{\mathrm{div}}
\newcommand{\ol}{\overline}

\newcommand{\eps}{\varepsilon}

\title{On the periodic homogenization of elliptic equations in non-divergence form with large drifts}

\author{Wenjia Jing \and Yiping Zhang}

\address{Yau Mathematical Sciences Center, Tsinghua University, Beijing 100084, and Yanqi Lake Beijing Institute of Mathematical Sciences and Applications, Beijing 101407, P.R. China}
\email{wjjing@tsinghua.edu.cn}

\address{Yanqi Lake Beijing Institute of Mathematical Sciences and Applications, Beijing 101407, and Yau Mathematical Sciences Center, Tsinghua University, Beijing 100084, P.R. China}
\email{zhangyiping161@mails.ucas.ac.cn}

\date{\today}

\begin{document}

\begin{abstract}

We study the quantitative homogenization of linear second order elliptic equations in non-divergence form with highly oscillating periodic diffusion coefficients and with large drifts, in the so-called ``centered'' setting where homogenization occurs and the large drifts contribute to the effective diffusivity. Using the centering condition and the invariant measures associated to the underlying diffusion process, we transform the equation into divergence form with modified diffusion coefficients but without drift. The latter is in the standard setting for which quantitative homogenization results have been developed systematically. An application of those results then yields quantitative estimates, such as the convergence rates and uniform Lipschitz regularity, for equations in non-divergence form with large drifts. 
\smallskip
 

\smallskip

\noindent{\bf Mathematics subject classification (MSC 2010)}: 35B27, 35J08

\end{abstract}

\maketitle


\section{Introduction and the main results}

In this paper, we investigate the periodic homogenization of linear second order elliptic equations in non-divergence form with large drift. More precisely, for $0<\varepsilon<1$, we consider the following equation
\begin{equation}
\label{eq:heteq}
\left\{
\begin{aligned}
&-{\tilde{a}}_{ij}(\textstyle\frac{x}{\eps}) \partial_i\partial_j u_\varepsilon(x) - \tfrac{1}{\varepsilon}\tilde{b}_i(\tfrac{x}{\eps})\partial_i u_\varepsilon(x) = f(x) \; &&\text{ in }\Omega,\\
&u_\varepsilon(x) = g(x) \; &&\text{ on }\partial\Omega.
\end{aligned}\right.
\end{equation}
Here and below, repeated indices are summed over their range unless otherwise stated, and $\bT^d := \R^d/\Z^d$ denotes the flat torus. $\Omega\subset \R^d$ is an open bounded domain. The following assumptions are imposed, for all $ i,j=1,\dots,d$:
\begin{equation}
\label{eq:Ab0Hypo}
\left\{
\begin{aligned}
  &\text{The domain $\Omega$ is of class $C^{1,1}$};\\
&\widetilde{a}_{ij}=\tilde{a}_{ji}, \; 
\widetilde{a}_{ij} \in C^1(\bT^d), \; \tilde b_i \in L^\infty(\bT^d); \\
&\widetilde{a}_{ij}(x+z) = \widetilde a_{ij}(x), \;\widetilde{b}_i(x+z) = \widetilde b_i(x), \; \forall z \in \Z^d, \forall x\in \R^d;\\ 
&\exists \, \lambda \in (0,\infty), \, \text{such that} \;  \lambda |\xi|^2 \leq \tilde{a}_{k\ell}(x)\xi_k\xi_\ell, \, \forall x,\xi \in \mathbb{R}^d.
\end{aligned}
\right.
\end{equation}
The second line implies $\|\tilde a_{ij}\|_{L^\infty(\bT^d)} +  \|\partial_\ell \tilde a_{\ell j}\|_{L^\infty(\bT^d)} + \|\tilde b_i\|_{L^\infty(\bT^d)} \le \Lambda$ for some $\Lambda \in (0,\infty)$ . Under the above assumptions, the diffusion matrix $\widetilde a$ is symmetric and uniformly elliptic and, together with the drift coefficient $\tilde b$, it is $\Z^d$-periodic and sufficiently regular. To emphasize the periodicity of the coefficients, we view $\tilde a, \tilde b$ as functions on $\bT^d$. The $C^{1,1}$ regularity for $\Omega$ is the standard setting for the Dirichlet problem \eqref{eq:heteq} for elliptic equations in non-divergence form, since it yields a strong solution $u^\eps \in H^2(\Omega)$ for each fixed $\eps$, $f\in L^2(\Omega)$ and $g\in H^2(\Omega)$. For the method of this paper to work out, the regularity assumptions for $\tilde a$ can be relaxed to $\tilde a \in C^{0,\delta}(\bT^d)$ for some $\delta \in (0,1)$. See Remarks \ref{rem:reg} and \ref{rem:abHolder} below. 
Less regular domains (e.g., bounded and convex, or piecewise regular Lipschitz domains, such as polyhedrons in $\R^3$) are often used in numerical methods for \eqref{eq:heteq}; see analysis of such domains in \cite{MR3985933,MR3077903,MR4520025} under certain restrictions of $\tilde a$. Our method applies to those settings; see Remark \ref{rem:C11} below. For presentational simplicity, however, we assume \eqref{eq:Ab0Hypo} throughout the paper.

The following notations and conventions are used throughout the paper: $\mathbf{I}_d$ denotes the $d\times d$ identity matrix. A bounding constant in an estimate is called \emph{universal} if it depends only on the quantities $d,\lambda,\Lambda,\Omega$ in \eqref{eq:Ab0Hypo} but is independent of $\eps$, $f$ and $g$. As usual, bounding constants in various lines may change but are denoted by the same notation.

The problem \eqref{eq:heteq} arises naturally when modeling diffusive phenomena in heterogeneous environments. It is the simplest model of this kind, a periodic one, but it reveals important features that are shared by more general situations. Let us define the differential operator
\begin{equation*}
L_{\tilde a, \tilde b}^\eps := -\tilde a_{ij}(\tfrac{x}{\eps})\partial_i \partial_j - \tfrac{1}{\eps} \tilde b_j(\tfrac{x}{\eps})\partial_j.
\end{equation*}
Using the regularity of $\tilde a$ we can rewrite $L^\eps_{\tilde a,\tilde b}$ as
  \begin{equation}
  \label{eq:heteqdiv}
  \begin{aligned}
  &L^\eps_{\tilde a, \tilde b}\,u_\eps = \mathcal{L}^\eps_{\tilde a, \tilde \beta}\, u_\eps := -\partial_i\left(\tilde a_{ij}(\tfrac{x}{\eps})\partial_j u_\eps \right) - \tfrac{1}{\eps}\tilde\beta_j(\tfrac{x}{\eps}) \partial_j u_\eps,\\
  &\text{where}\; \tilde \beta_j(y) := \tilde b_j(y) - \partial_{y_i} \tilde a_{ij}(y). 
  \end{aligned}
\end{equation}
The differential operator $\mathcal{L}^\eps_{\tilde a,\tilde \beta}$ is in divergence form and still with a large drift. 

The differential operator $L^\eps_{\tilde a, \tilde b}$ is the generator of the diffusion process determined by the following stochastic differential equation (SDE):
\begin{equation}
\label{eq:SDE}
\left\{
\begin{aligned}
dX^\eps_t &= \tfrac{1}{\eps} \tilde b(\tfrac{X^\eps_t}{\eps}) dt + \sqrt{2} \sigma(\tfrac{X^\eps_t}{\eps}) dW_t,\\
X^\eps_0 &= x.
\end{aligned}
\right.
\end{equation}
Here, $\sigma(x) = \sqrt{\tilde a(x)}$ is the square root of the positive definite matrix $\tilde a(x)$ and $W_t$ is a standard $d$-dimensional Wiener process. Via a standard change of variable one checks that the law of $X^\eps_t$ is the same to $\eps X_{t/\eps^2}$ where $(X_s)_{s\in \R_+}$ is determined by
\begin{equation*}
dX_s = \tilde b(X_s)ds + \sqrt{2}\sigma(X_s) dW_s, \qquad X_0 = \tfrac{x}{\eps}.
\end{equation*}
Under the periodicity assumptions of the coefficients $\tilde a, \tilde b$, the path $X_s$ can be viewed as living in the torus $\bT^d$. 
In view of this connection between \eqref{eq:heteq} and SDEs, there is a probabilistic approach to study the homogenization problem as done in the seminal work \cite{MR503330} by Bensoussan, Lions and Papanicolaou; see Chapter 3 there. Under proper conditions (see \eqref{eq:bcenter}) on the drift $\tilde b$, the solution $u_\eps$ is known to converge weakly in $H^1(\Omega)$ to the solution $u$ of a homogenized problem with constant diffusion coefficients with no drift. In other words, the original large drift contributes to the effective diffusion in a spatial scale much larger than the periodicity of it. 

This paper is mainly concerned with quantitative aspects of the homogenization of \eqref{eq:heteq}. The authors of \cite{MR503330} obtained an $L^\infty$ convergence rate using the classical two-scale expansion method. However, their method requires higher (than \eqref{eq:Ab0Hypo}) regularities on $\tilde a, \tilde b$ and on $f$ and only treats the case of $g=0$. New ideas and techniques for quantitative homogenization, not only in the periodic but also in the stationary ergodic settings, have been developed in the recent decades. It is natural to check how such advances apply to \eqref{eq:heteq}. We refer to \cite{MR503330,MR1329546}, as representative works, for the classical theory on homogenization, and to \cite{MR3225629,MR3082248,MR3073881,MR3648977,MR4377865,MR4103433,MR4567768,MR3838419} for recent developments with emphases on the quantitative aspects. 

In periodic homogenization there are now standard methods (see e.g.\,\cite{MR3838419}) to obtain (even optimal) convergence rates in $L^p$ and $W^{1,p}$ (with proper correctors) etc., to describe the asymptotic behaviours of the Green functions, and to prove regularity estimates that are uniform in $\eps$. For \eqref{eq:heteq}, in view of the connection to \eqref{eq:SDE}, it is most natural to consider $L^\eps_{\tilde a,\tilde b}$ in non-divergence form when the diffusion coefficient $\tilde a$ is not a constant matrix. However, most works about quantitative homogenization concentrate on equations in divergence form rather than the form of \eqref{eq:heteq};  even when the (weak) differentiability of $\tilde a$ is imposed so that we can rewrite $L^\eps_{\tilde a,\tilde b}$ into divergence form, as in \eqref{eq:heteqdiv}, much fewer results are available when the large drift $\eps^{-1}\tilde\beta(x/\eps)$ is present. Nevertheless, see Allaire and Orive \cite{MR2351401}, Henning and Ohlberger \cite{MR2819498} and Capdeboscq \cite{MR1663726} for qualitative homogenization results in this (and a bit more general) setting. The main objective of this paper is to study quantitative homogenization results when the large drift is present.

When $\tilde{b}=0$ and in the periodic setting, Avellaneda and Lin \cite{MR978702} obtained various uniform regularity estimates for $u_\eps$ (up to the class of $C^{1,1}$ in certain settings) using their influential compactness method. For convergence rates, Guo, Tran and Yu \cite{MR4354730} showed that $O(\eps)$ is the generic optimal rate in $L^\infty$, and they constructed correctors to get an $O(\eps^2)$ rate in $L^\infty$. They also initiated the study of under what conditions the convergence rate is $O(\eps^2)$ without using correctors; see \cite{https://doi.org/10.48550/arxiv.2201.01974} for such studies and see Sprekeler and Tran \cite{MR4308690} for $O(\eps)$ convergence rates in $W^{1,p}$. We refer to \cite{MR3269637,MR3665674} for uniform regularity results in the random setting with short range dependence assumptions, and to \cite{MR4491712,https://doi.org/10.48550/arxiv.2301.01267} for the studies from the random walk in a random environment point of view.

Numerical approximations of the solution to the equation \eqref{eq:heteq} in non-divergence form with highly oscillating coefficients is both challenging and of practical importance, and it inspires much analysis of \eqref{eq:heteq}. We refer to Capdeboscq, Sprekeler and S\"uli \cite{MR4111657}, Henning and Ohlberger \cite{MR2740530} and the references therein for more details.  

In the following, we first review the qualitative theory, and then state our main results on the quantitative homogenization. 

\subsection{The qualitative homogenization result}
\label{sec:qual}

Under the conditions in \eqref{eq:Ab0Hypo}, it is known (see \cite[Theorem 3.4 of Chapter 3]{MR503330} and Proposition \ref{prop:invmeas} below) that there exists a unique invariant measure with positive density $m(y) \in C(\bT^d)$ for the diffusion process \eqref{eq:SDE}, and $m$ is the unique weak solution to the equation:
\begin{equation}
\label{eq:invmeas}
-\partial_{y_i}\left[ \partial_{y_j}\left(\tilde{a}_{ij}(y) m(y)\right)-  \tilde{b}_i(y)m(y) \right]=0  \text{ in }\bT^d, \text{ and }
\int_{\bT^d} m(y)\, dy=1.
\end{equation}
Let $\tilde \beta$ be defined as in \eqref{eq:heteqdiv}. The equation for $m$ can also be put in divergence form as
\begin{equation}
\label{eq:invmeas-div}
-\partial_{y_i}\left[\tilde a_{ij}(y) \partial_{y_j} m(y) - \widetilde \beta_i(y)m(y) \right] = 0.
\end{equation}
See Proposition \ref{prop:invmeas} below for more details. In \cite{MR503330} the qualitative homogenization of \eqref{eq:heteq} was established under the following additional condition:
\begin{equation}
\label{eq:bcenter}
\int_{\bT^d}  \tilde{b}_j(y)m(y)\, dy=0, \qquad j=1,\dots,d.
\end{equation}
The above is henceforth referred to as the \emph{centering condition}. Since the existence and uniqueness of $m$ is guaranteed by the assumptions in \eqref{eq:Ab0Hypo}, the problem \eqref{eq:invmeas} with \eqref{eq:bcenter} form an overdetermined system which has a solution only for some special class of drifts. 

In \cite{MR503330}, Bensoussan, Lions and Papanicolaou established the homogenization of \eqref{eq:heteq} using both probabilistic and PDE based analytic methods. In both approaches, the centering condition \eqref{eq:bcenter} is the key assumption. It allows one to solve the following \emph{cell problem} (a central concept for homogenization). More precisely, \eqref{eq:Ab0Hypo} and \eqref{eq:bcenter} guarantee, for each $j=1,\cdots,d$, the unique existence of $\tilde{\chi}^j$ which solves
\begin{equation}
\label{eq:cellprob}
-{\tilde{a}}_{ik}\partial_{y_i}\partial_{y_k} \tilde{\chi}^j(y) - \tilde{b}_i(y)\partial_{y_i} \tilde{\chi}^j(y)=\tilde{b}_j(y) \; \text{ in }\bT^d, \quad \text{and}\quad 
\int_{\bT^d} \tilde{\chi}^j=0.
\end{equation}
The qualitative homogenization result is then as follows.
\begin{theorem}[Theorem 3.5.2 in \cite{MR503330}]
\label{thm:homog}
Suppose that \eqref{eq:Ab0Hypo} and \eqref{eq:bcenter} hold. Define $\ol a = (\ol a_{ij})$ by
\begin{equation}
\label{eq:abar}
\begin{aligned}
\ol a_{ij} := &\int_{\bT^d} [(I + \nabla_y \widetilde\chi)\widetilde a (I+\nabla_y \widetilde\chi)^{\rm T}(y)]_{ij} m(y)\, dy\\
 = &\int_{\bT^d} \left(\widetilde a_{ij} + \widetilde a_{ik}\partial_{y_k}\widetilde\chi^j + \widetilde a_{kj} \partial_{y_k} \widetilde\chi^i  + \widetilde a_{k\ell}  \partial_{y_k} \widetilde\chi^i \partial_{y_\ell} \widetilde\chi^j\right) m(y)\, dy.
 \end{aligned}
\end{equation}
Then $\ol a$ is a constant symmetric matrix, and $\ol a \ge \lambda_1 \mathbf{I}_d$ for some positive constant $\lambda_1 > 0$ that is universal. Moreover, for any $f \in L^2(\Omega)$ and $g\in H^2(\Omega)$ (that is, the Dirichlet datum is the restriction on $\partial \Omega$ of such a function), the solution $u_\eps$ of \eqref{eq:heteq} converges weakly in $H^1(\Omega)$ to the solution $u$ of the homogenized problem
\begin{equation}
\label{eq:homogeq}
\left\{\begin{aligned}
&-\ol a_{ij}\partial_{i}\partial_{j} u = f \quad &&\text{ in }\Omega,\\
&u = g \quad &&\text{ on } \partial \Omega.
\end{aligned}\right.
\end{equation}
\end{theorem}

This theorem appeared as Theorem 5.2 in Chapter 3 of \cite{MR503330} under stronger assumptions than ours. We reprove this theorem in Section \ref{sec:qualproof}. 
It is important to note that under the $C^{1,1}$ regularity assumption of $\Omega$, the solution of \eqref{eq:homogeq} actually belongs to $H^2(\Omega)$. 

\begin{remark}\label{rem:1D}
A discussion about the centering condition is now in order. 

If the large drift is not present, i.e., $\tilde b = 0$, then the centering assumption is always satisfied. This is the case in \cite{MR978702,MR4354730,MR4308690,MR3665674}. If the matrix $\tilde a$ is constant and $\nabla\cdot \tilde b = 0$, then $m(y) \equiv 1$ is the invariant measure and the centering condition reduces to the mean-zero condition; detailed studies of such cases can be found in \cite{MR1265233}.  

In the so-called laminated media where the coefficients $\tilde a$ and $\tilde b$ only depend on one coordinate, namely the first one $y_1$ of $y = (y_1,\dots,y_d)$, the invariant measure $m(y)$ is of the form $m(y) = m(y_1)$ and it is determined by
\begin{equation*}
-\partial_{y_1}^2 (\tilde a_{11}(y_1)m(y_1)) + \partial_{y_1} (\tilde b_1(y_1)m(y_1)) = 0, \qquad y_1 \in \bT^1.
\end{equation*}
Note that $\tilde a_{11} > 0$ due to ellipticity. Assuming that $\tilde a, \tilde b$ are sufficiently smooth, the centering condition is then equivalent to
\begin{equation}
\label{eq:1Dcenter}
\int_0^1 \frac{\tilde b_1(s)}{\tilde a_{11}(s)} ds = 0, \; \int_0^1 \frac{\tilde b_j(s)}{\tilde a_{11}(s)} \exp\left(\int_0^s \frac{\tilde b_1(t)}{\tilde a_{11}(t)}\,dt \right)ds = 0, \text{ for } j=2,\dots,d.
\end{equation}
The simpler 1D case is treated in more detail in Section \ref{sec:1D}. We also remark there that one cannot expect to have homogenization in general when the centering condition fails. 
\end{remark}
\subsection{Main results on quantitative estimates}
\label{sec:quant}

Our main results of this paper concern the quantitative estimates for the homogenization of \eqref{eq:heteq}, namely the convergence rates in $L^2$, $L^\infty$ and in $H^1$ and the uniform Lipschitz regularity of $\{u_\eps\}_{\eps\in (0,1)}$.

As is standard, for convergence rates in $H^1$ (in general for $W^{1,p}$) some correctors are needed (add to the limit $u$). Following \cite{MR3082248,MR3838419}, we introduce the so-called \emph{Dirichlet correctors} $\Phi_{\eps,j}$, $j=1,\dots,d$, defined by
\begin{equation}
\label{eq:Dirichletcorrector}
-\partial_i\left(q_{ik}(\tfrac{x}{\eps})\partial_k \Phi_{\eps,j}\right) = 0 \; \text{in } \Omega, \qquad \Phi_{\eps,j} = x_j \; \text{on } \partial \Omega.
\end{equation}
Here, the diffusion matrix $q = (q_{ij})$ is defined later in \eqref{eq:qijdef} and is uniformly elliptic. It is easily verified that the function $\Phi_{\eps,j} - x_j$ belongs to $H^1_0(\Omega)$ and has size of order $O(\eps)$ in $L^\infty$. The size of its gradient, however, is not small in $L^2$. This function corrects $\nabla u_\eps - \nabla u$ to make the latter converge strongly.

The first of our main results is concerned with the quantification of the convergence of $u_\eps$ to $u$ in various functional spaces.
\begin{theorem}
  \label{thm:rates}
  Assume that \eqref{eq:Ab0Hypo} and \eqref{eq:bcenter} hold. Then the following results are true.
  \begin{itemize}
    \item[(1)] There exists a universal constant $C\in (0,\infty)$ such that for any $f\in H^1(\Omega)$ and $g \in H^2(\Omega)$, we have
  \begin{equation}
    \label{eq:L2rate}
    \|u_\varepsilon-u\|_{L^2(\Omega)}\leq C\eps \left(\|f\|_{H^1(\Omega)} + \|g\|_{H^2(\Omega)}\right)
    \end{equation} and
\begin{equation}
    \label{eq:H1rate}
    \|u_\varepsilon-u - \{\Phi_{\eps,j}-x_j\}\partial_j u\|_{H^1_0(\Omega)}\leq C\eps \left(\|f\|_{H^1(\Omega)} + \|g\|_{H^2(\Omega)}\right).
    \end{equation}
  \item[(2)] Let $g=0$ and $p\in (1,d)$, and let $r=dp/(d-p)$. Then there is a constant $C_p \in (0,\infty)$ depending only on the data in \eqref{eq:Ab0Hypo} and on $p$ so that, for any $f\in W^{1,p}(\Omega)$, we have
    \begin{equation}
    \label{eq:Lqrate}
    \|u_\eps - u\|_{L^r(\Omega)} \le C_p\,\eps \|f\|_{W^{1,p}(\Omega)}.
    \end{equation}
  \item[(3)] Let $g=0$ and $p\in (d,\infty)$. Then there is a constant $C_p \in (0,\infty)$ depending only on the data in \eqref{eq:Ab0Hypo} and on $p$ so that, for any $f\in W^{1,p}(\Omega)$, we have
    \begin{equation}
    \label{eq:Linfrate}
    \|u_\eps - u\|_{L^\infty(\Omega)} \le C_p\,\eps \|f\|_{W^{1,p}(\Omega)}.
    \end{equation}
  \end{itemize}
\end{theorem}

We also have the following uniform (in $\eps$) Lipschitz regularity for the solutions $\{u_\eps\}_\eps$ to \eqref{eq:heteq}.

\begin{theorem}
  \label{thm:Lipreg}
  Assume that \eqref{eq:Ab0Hypo} and \eqref{eq:bcenter} hold. Let $p>d$ and $\eta \in (0,1)$ be fixed numbers. Then there exists a constant $C_{p,\eta} \in (0,\infty)$ depending only on the data in \eqref{eq:Ab0Hypo} and on $p,\eta$, so that, for any $f\in L^p(\Omega)$ and $g\in C^{1,\eta}(\partial \Omega)$, we have 
\begin{equation}
  \label{eq:Lipreg}
  \|\nabla u_\varepsilon \|_{L^\infty (\Omega)} \leq C_{p,\eta}\,\{\|g\|_{C^{1,\eta}(\partial \Omega)} + \|f\|_{L^p(\Omega)}\}.
\end{equation}
\end{theorem}

\smallskip

Item (3) of Theorem \ref{thm:rates} recovers the $L^\infty$ convergence rate as in Theorem 5.1 of Chapter 3 in \cite{MR503330}. All other results concerning the convergence rates and the uniform Lipschitz estimate above are new in the setting of this paper. Note also, in view of the uniform Lipschitz regularity of $\{u_\eps\}_\eps$, we also recover the convergence of $u_\eps$ to $u$ in $C(\ol\Omega)$ (with same rate as \eqref{eq:Linfrate}). The qualitative convergence in $C(\ol\Omega)$ was established in Theorem 4.5 of Chapter 3 in \cite{MR503330} using probabilistic methods. Moreover, the $L^\infty$ rate in \cite{MR503330} was established under higher regularity assumptions on $\tilde a, \tilde b$ and for $f\in W^{3,p}(\Omega)$ with $p>d$. Our result, hence, is an improvement. 

On the other hand, all those quantitative results look almost identical to the corresponding results recently developed for periodic homogenization of elliptic equations in divergence form without any drifts, namely in \cite{MR3225629,MR3073881,MR3838419}, except for the higher regularity requirement in $f$ in Theorem \ref{thm:rates}. In fact, the main contribution of this paper is the observation that, under the centering condition \eqref{eq:bcenter}, \eqref{eq:heteq} can be transformed into an elliptic equation with periodic and uniformly elliptic coefficients without any drift; see \eqref{eq:heteq2} below. This allows us to use the recent quantitative homogenization results in the more standard setting directly to get the results above. It will be clear from the proof that the higher regularity requirement on $f$ is also due to this transformation. 


\begin{remark}\label{rem:reg}
We finish the introduction by several remarks.

The key transformation that puts \eqref{eq:heteq} into a divergence form equation without drift is carried out in detail in the next section. For $\tilde b = 0$, such a transformation was used already by Avellaneda and Lin \cite{MR978702}. We show in this paper that it works for more general $\tilde b$ that satisfies the centering condition. If this last condition fails, the homogenization is more involved and one cannot expect \eqref{eq:homogeq} as the effective model; see Section \ref{sec:1D} and the works in \cite{MR2819498,MR2324490,MR1663726}. 

The Lipschitz class is the sharp space for uniform regularity of the solutions to \eqref{eq:heteq}. In general, we cannot expect to obtain uniform regularity in $C^{1,r}(\Omega)$ for $r>0$; see discussions in Section \ref{sec:1D}. This is a clear contrast with the case of $\tilde b = 0$. For the latter setting, uniform $C^{1,1}$ a priori estimates were established by Avellaneda and Lin in \cite{MR978702}.  

It is possible to relax the regularity assumptions in \eqref{eq:Ab0Hypo}. As long as the regularity for $\tilde{a}$ is concerned, the key is to have the existence and uniqueness of the invariant measure $m$ and to make sure that $m\in C^{0,\alpha}(\bT^d)$ for some $\alpha \in (0,1)$. Hence, it is enough to require $\tilde a \in W^{1,p}(\bT^d)$ (for $\tilde b$, it suffices to impose $\tilde b\in L^p(\bT^d)$ for $p>d$ large enough). Under those assumptions, the existence and regularity of $m$ was studied systematically in \cite{MR4308690}. We also note that the proof about $m$ in this paper (see Proposition \ref{prop:invmeas} below) is essentially the one in \cite{MR503330} and it can be carried out under the quite weak regularity assumption $\tilde{a} \in C^{0,\delta}(\bT^d)$, $\beta\in (0,1)$ and $\tilde b\in L^\infty(\bT^d)$; see Remark \ref{rem:abHolder}. This assumption on $\tilde a$ is more or less optimal. 

Last but not least, the quantitative results selected in Theorems \ref{thm:rates} and \ref{thm:Lipreg} are just a few representatives. Various other results (e.g.\,$W^{1,p}$ rate, Neumann boundary conditions, etc.) in \cite{MR3838419} can also be considered here.
\end{remark}

The rest of the paper is organized as follows. In the next section we use the key transformation to put \eqref{eq:heteq} into a divergence form equation without any drift. The resulting equation is in the standard form for which the recently developed quantitative homogenization results apply. We then apply those results to prove the main theorems of this paper in Section \ref{sec:proofs}. In sections \ref{sec:1D} and \ref{sec:conclusions} we comment on the centering conditions for the drifts, provide some examples and further discussions.

\section{The key transformation}
\label{sec:transformation}

In this section we transform \eqref{eq:heteq} into an equation in divergence form without any drift term. In the case of $\tilde b = 0$, this transformation was already used in Avellaneda and Lin \cite{MR978702}. It consists of two steps as follows. 

\subsection{Weighting by the invariant measure}

In the first step, we weight the equation \eqref{eq:heteq} by the invariant measure $m$ and change the equation into divergence form with a large drift that is mean-zero and divergence free. Note that, without this weighting, the drift $\tilde \beta$ in \eqref{eq:heteqdiv} does not have those properties.

First, the invariant measure $m$ in \eqref{eq:invmeas} is well defined with the following important properties.
\begin{proposition}\label{prop:invmeas} Under the assumptions in \eqref{eq:Ab0Hypo}, the equation \eqref{eq:invmeas} admits a unique weak solution $m \in H^1(\bT^d)$. Moreover, we can find $\alpha \in (0,1)$ and $C\in (1,\infty)$, both of which are universal, such that $m\in C^{0,\alpha}(\bT^d)$ and
\begin{equation}
\label{eq:mbound}
\|m\|_{C^{0,\alpha}(\bT^d)} + \|m\|_{H^1(\bT^d)} \le C, \quad \inf_{y\in \bT^d} m(y) \ge C^{-1}. 
\end{equation}
\end{proposition}

\begin{proof} The existence and uniqueness of $m\in H^1(\bT^d)$ that solves \eqref{eq:invmeas} follows essentially from the proof of Theorem  3.4 in Chapter 3 of \cite{MR503330}. The proof there is based on a Fredholm alternative argument and the key is to show that if $z\in H^1(\bT^d)$ solves
	\begin{equation*}
	  -\partial_{y_i}\left(\tilde a_{ij}(y) \partial_{y_j} z(y)\right) -\tilde \beta_i(y)\partial_{y_i} z(y) = 0 \quad \text{in $\bT^d = \R^d/\Z^d$}, 
	\end{equation*}
then $z$ must be a constant. In \cite{MR503330} this was proved for $\tilde a\in C^1\cap W^{2,\infty}(\bT^d)$ and $\tilde b\in C^1(\bT^d)$ with the by-product that $m\in W^{2,p}(\bT^d)$ for some $p>2$. Inspecting the proof, however, we see that the above holds under the assumption \eqref{eq:Ab0Hypo}. Indeed, by elliptic regularity  (e.g.\,Theorem 8.24 of \cite{MR1814364}) $z\in C^{0,\alpha}(\bT^d)$ for some $\alpha \in (0,1)$ and hence we can assume $z> 0$. Then we can conclude using the strong maximum principle (e.g.\,Theorem 8.19 of \cite{MR1814364}) for the equation above. Once we solve \eqref{eq:invmeas-div} for $m\in H^1(\bT^d)$, using elliptic regularity again we get $m\in C^{0,\alpha}(\bT^d)$ with the desired estimates. The existence of a positive lower bound was established, again, in \cite{MR503330}. 
\end{proof}

\begin{remark}\label{rem:abHolder} 
  The above proof is essentially from Bensoussan, Lions and Papanicolaou \cite{MR533346}. Using the more recent solvability and regularity theory of the double divergence form equation \eqref{eq:invmeas} by Bogachev and Shaposhnikov \cite[Corollary 3.7 and Theorem 3.1]{MR3694737} (and see also the work of Sj\"ogren \cite{MR346173}), one can still obtain a positive and H\"older continuous invariant measure $m$ after weakening the regularity of the diffusion matrix to $\tilde a\in C^{0,\delta}(\bT^d)$ for some $\delta \in (0,1)$. Then, although the modified drift $\tilde\beta$ defined in \eqref{eq:invmeas-div} is no longer in $L^\infty$, it is the sum of a bounded term with a weak divergence of a $C^{0,\alpha}$ term, and all the elliptic PDE tools can still be used. We omit further discussions on this.
  
In the case of $\tilde b=0$, Sprekeler \cite{sprekeler2023homogenization} obtained existence and uniqueness of a non-negative invariant measure $m\in L^2(\bT^d)$ associated to a uniformly elliptic diffusion matrix $\tilde a \in L^\infty(\bT^d)$ (and under the Cordes condition if $d\ge 3$). Among quite a few other interesting results, he used $m$ and the transformation method of \cite{MR978702} to establish homogenization of \eqref{eq:heteq} (weakly in $H^2$ and strongly in $H^1$) for $\Omega$ that is bounded and convex. 
\end{remark}

We put weights on the coefficients and the right hand side of \eqref{eq:heteq} and, for $y\in \bT^d$ and $x\in \Omega$, define:
\begin{equation}
\label{eq:abtrans}
a_{ij}(y) = \widetilde a_{ij}(y) m(y), \; b_j (y) = \widetilde b_j (y) m(y), \;  f_\eps(x) = f(x)m(\tfrac{x}{\eps}).
\end{equation}
Then problem \eqref{eq:heteq} can be rewritten as
\begin{equation}
\label{eq:heteq1}
\left\{\begin{aligned}
&-\partial_i\left({a}_{ij}(\tfrac{x}{\eps}) \partial_{j}u_\varepsilon\right)
- \tfrac{1}{\varepsilon}{\beta}_i(\tfrac{x}{\eps}) 
\partial_i u_\varepsilon = f_\varepsilon &&\text{ in }\Omega,\\
&u_\varepsilon=g &&\text{ on }\partial\Omega.
\end{aligned}\right.
\end{equation}
Here the periodic vector field $\beta = (\beta_j)$, $j = 1,\dots,d$, is defined by 
\begin{equation}
\label{eq:betavec}
{\beta}_j (y) = b_j(y) - \partial_{y_i}{a}_{ij}(y)  = \widetilde b_j(y) m(y) - \partial_{y_i} \left(\widetilde a_{ij}(y) m(y)\right), \qquad y\in \bT^d.
\end{equation}
In view of the regularity properties of $\tilde a, \tilde b$ and $m$, we check that $b_j\in L^\infty(\bT^d)$ and $a_{ij} \in C^{0,\alpha}(\bT^d)$ for all $i,j = 1,\dots,d$. 
Moreover, by \eqref{eq:invmeas} we obtain the following key property for $\beta$:
\begin{equation}
\label{eq:betakey}
\partial_{y_i}{\beta}_i=0,\qquad \int_{\bT^d} {\beta}_i(y)dy=0.
\end{equation}

We point out that in the PDE method for qualitative homogenization in \cite{MR503330}, the authors there started from \eqref{eq:heteq1} and adapted the usual energy method by paying extra attention to the large drift term. In particular, a formal two-scale expansion suggests one to consider the following cell problem: for each $j=1,\dots,d$,
\begin{equation}
\label{eq:celleq1}
-\partial_{y_k}(a_{k\ell}(y)\partial_{y_\ell} \chi^j(y)) - \beta_\ell(y)\partial_{y_\ell} \chi^j(y) = \partial_{y_k} a_{kj}(y) + \beta_j(y)  \text{ in }  \bT^d, \quad \int_{\bT^d} \chi^j = 0.
\end{equation}
The fact that the drift term satisfies \eqref{eq:betakey} is crucial, since it guarantees the unique solvability of \eqref{eq:celleq1}. See Chapter 3 of \cite{MR503330} for details.  

\begin{remark}\label{rem:quant1} Let us also point out the following: One could continue the PDE approach of \cite{MR503330} and quantify the homogenization of \eqref{eq:heteq1} directly, by adapting the recently developed methods for quantitative results (e.g.\,in \cite{MR3838419}) for divergence form and by carefully tracking the effects of the large drift. Roughly speaking, it suffices to replace the cell problems in the standard setting (see \cite{MR3838419}) by \eqref{eq:celleq1}. 
\end{remark}

\subsection{A further transformation for the drift term}
\label{sec:fluxtran}

We take a much simpler approach than the one outlined in the remark above. This is the second step of the key transformation. 

Recall that due to the centering condition of $\tilde b$, the large drift in the resulted equation \eqref{eq:heteq1} is mean zero and divergence free. The following then holds.
\begin{lemma}
\label{lem:flux}
Assume \eqref{eq:Ab0Hypo} and \eqref{eq:bcenter}. Let $\beta$ be as in \eqref{eq:betavec} and $\alpha \in (0,1)$ as in Proposition \ref{prop:invmeas}. There exists an anti-symmetric 2-tensor $\phi = (\phi_{ij})$ and a universal constant $C\in (1,\infty)$ so that $\phi_{ij} \in C^{0,\alpha}(\bT^d)$, and for all $k,j=1,\dots,d$, the following holds:
\begin{equation}
\label{eq:phi2tensor}
\beta_j =\partial_{y_\ell}\phi_{\ell j},\ \phi_{kj}=-\phi_{jk}, \quad \int_{\bT^d} \phi_{kj}(y)\,dy=0,\, \text{and } \, 
\|\phi_{kj}\|_{C^{0,\alpha}(\bT^d)} \leq C. 
\end{equation}
\end{lemma}

Results of this type play important roles in homogenization theory and they were present in classical books like \cite{MR503330,MR1195131}; see also \cite[Section 3.1]{MR3838419}. We provide some details of the proof below for the sake of completeness. 

\begin{proof} For each $j=1,\dots,d$, solve the Poisson problem
\begin{equation*}
\Delta_y f^j = \beta_j =  b_j - \partial_{y_i} a_{ij} \quad \text{ in }\, \bT^d, \qquad \int_{\bT^d} f^j(y)\, dy = 0.
\end{equation*} 
Since $b_j\in L^\infty(\bT^d)$ and $a_{ij} \in C^{0,\alpha}(\bT^d)$, by elliptic regularity theory we get $f^j \in C^{1,\alpha}(\bT^d)$. Let
\begin{equation*}
\phi_{ij} := \partial_{y_i} f^j - \partial_{y_j} f^i.
\end{equation*}
Then $\phi_{ij} \in C^{0,\alpha}(\bT^d)$ and it clearly satisfies $\phi_{ij} = -\phi_{ji}$. The identity $\partial_{y_i} \phi_{ij}(y) = \beta_j$, which is equivalent to $\partial_{y_i}\partial_{y_j} f^i(y) = 0$, can be checked by verifying that $\Delta_y (\div f) = 0$. The latter follows from the equations of $(f^j)$ and the fact that $\div\, \beta = 0$. The estimate in \eqref{eq:phi2tensor} follows from the definitions of $\beta,\phi$ and from the bounds in \eqref{eq:Ab0Hypo} and \eqref{eq:mbound}.
\end{proof}

The second step of our key transformation is carried out as follows. Define
\begin{equation}
\label{eq:qijdef}
q_{ij}(y) = a_{ij}(y) + \phi_{ij}(y) = \widetilde a_{ij}(y) m(y) + \phi_{ij}(y), \qquad y\in \bT^d, \, i,j=1,\dots,d.
\end{equation}
The following is a direct consequence of the previous lemma.
\begin{corollary} \label{cor:q}
The diffusion matrix $q = (q_{ij})$ is in $C^{0,\alpha}(\bT^d)$ and is uniformly elliptic, and there exist universal constants $\lambda_1,\Lambda_1 \in (0,\infty)$ so that 
\begin{equation}
\label{eq:qelliptic}
\|q_{ij}\|_{L^\infty(\bT^d)} \le \Lambda_1, \quad q_{ij}(y)\xi^i\xi^j \ge \lambda_1 |\xi|^2, \; \forall y\in \bT^d, \,\forall \xi \in \R^d.
\end{equation}
Moreover, the problem \eqref{eq:heteq1} can be rewritten as
\begin{equation}
\label{eq:heteq2}
\left\{\begin{aligned}
&-\partial_i\left(q_{ij}(\tfrac{x}{\eps}) \partial_{j}u_\varepsilon\right)  = f_\varepsilon &&\text{ in }\Omega,\\
&u_\varepsilon=g &&\text{ on }\partial\Omega.
\end{aligned}\right.
\end{equation}
\end{corollary} 
Let us point out that this second step of transformation was used in \cite{MR1663726} and is now standard.

\begin{proof} In view of the regularity of $\phi_{ij}$, the bounds and the positivity of $m$, and the anti-symmetry of $(\phi_{ij})$, we check that
\begin{equation*}
\lambda_1 := \lambda\,\inf_{\bT^d} m, \quad \Lambda_1 := \max_{i,j}\|\phi_{ij}\|_{L^\infty} + \Lambda \, \max_{\bT^d} m.
\end{equation*}
works for \eqref{eq:qelliptic}. To check the equivalence between \eqref{eq:heteq1} and \eqref{eq:heteq2} it suffices to verify
\begin{equation}
\label{eq:cor:heteq12}
-\int_\Omega \left[\tfrac{1}{\eps} \beta_j (\tfrac{x}{\eps})\partial_j u_\eps(x)\right] \varphi(x)\, dx = \int_\Omega \phi_{ij}(\tfrac{x}{\eps}) \partial_j u_\eps(x) \partial_i \varphi(x)\, dx, \qquad \forall \varphi \in C^\infty_c(\Omega).
\end{equation}
To this end, using the relation 
\begin{equation*}
\partial_\ell [\phi_{\ell j}(\tfrac{x}{\eps})] = (\eps^{-1} \partial_\ell \phi_{\ell j})(\tfrac{x}{\eps}) = \eps^{-1}\beta_j(\tfrac{x}{\eps})
\end{equation*}
we can compute the left hand side of \eqref{eq:cor:heteq12} as follows:
\begin{equation*}
\begin{aligned}
&-\int_\Omega [\partial_\ell(\phi_{\ell j}(\tfrac{x}{\eps}))]\varphi(x)\partial_j u_\eps(x)\,dx = \int_\Omega \phi_{\ell j}(\tfrac{x}{\eps}) \partial_\ell \left[\varphi \partial_j u_\eps\right]\,dx\\ 
= &\int_\Omega \phi_{\ell j}(\tfrac{x}{\eps}) (\partial_\ell \varphi) \partial_j u_\eps\,dx + \int_\Omega \varphi(x) \phi_{\ell j}(\tfrac{x}{\eps})\partial_{j}\partial_{\ell } u_\eps(x)\, dx.
\end{aligned}
\end{equation*}
The last term in the second line vanishes because $\phi$ is anti-symmetric (note that $u_\eps \in H^2(\Omega)$ for each $\eps \in (0,1)$ and $f\in L^2(\Omega)$). This verifies \eqref{eq:cor:heteq12} and completes the proof of the corollary.
\end{proof}

To summarize, we have transformed \eqref{eq:heteq} into the elliptic equation \eqref{eq:heteq2} which is in divergence form and without any drift term. Moreover, the unscaled diffusion matrix $q = (q_{ij})$ is $\Z^d$-periodic, uniformly elliptic and belongs to $C^{0,\alpha}(\bT^d)$. In other words, \eqref{eq:heteq2} is in the standard setting for quantitative periodic homogenization results; see \cite{MR3838419}. Let us emphasize that this two-step transformation was already used in Avellaneda and Lin \cite{MR978702} when $\tilde b = 0$. We provide above the details of the transformation for non-zero $\tilde b$ that satisfies the centering condition \eqref{eq:bcenter}.

\section{Proofs of the main results}
\label{sec:proofs}

With the preparations in the last section, we can apply the now well known quantitative homogenization results to \eqref{eq:heteq2} and prove the main theorems of the paper.

\subsection{The qualitative convergence result}
\label{sec:qualproof}

First, we reprove Theorem \ref{thm:homog} using the equivalence between \eqref{eq:heteq} and \eqref{eq:heteq2}. For the latter equation, since $(q_{ij})$ is $\Z^d$-periodic and satisfies \eqref{eq:qelliptic}, and
\begin{equation*}
f_\eps(x) = f(x)m\left(\tfrac{x}{\eps}\right) \rightharpoonup f(x)\int_{\bT^d} m(y)\, dy = f(x) \qquad \text{in } \, L^2(\Omega),
\end{equation*}
by the standard qualitative homogenization theory, we know that $u_\eps$ converges weakly in $H^1(\Omega)$ to the unique solution of 
\begin{equation}
  \label{eq:homog-q}
\left\{
    \begin{aligned}
	  &-\partial_{i}\left(\ol q_{ij} \partial_{j} u\right) = f, \quad &&\text{in } \Omega,\\
    &u = g \quad &&\text{on } \partial \Omega.
    \end{aligned}
\right.
\end{equation}
Here, the homogenized diffusion matrix $\ol q = (\ol q_{ij})$ is given by (see e.g. \cite{MR503330,MR1329546}):
\begin{equation*}
\ol q = \int_{\bT^d} (I+\nabla_y \chi(y)) q(y) (I+\nabla_y \chi(y))^{\rm T}\, dy
\end{equation*}
where $\chi = (\chi^1,\dots,\chi^d)$, and, for each $j=1,\dots,d$, $\chi^j \in H^1(\bT^d)$ is the solution of the corresponding cell problem
\begin{equation}
\label{eq:celleq2}
-\partial_{y_i}(q_{ik}(y) (\partial_{y_k} \chi^j + \delta_{kj})) = 0 \; \text{in } \bT^d, \quad \int_{\bT^d} \chi^j = 0.
\end{equation}
Note that the matrix $\ol q$ is not necessarily symmetric since $q$ is not. Under the $C^{1,1}$ regularity of $\Omega$, the unique $H^1$ solution $u$ of the homogenized equation \eqref{eq:homog-q} in fact belongs to $H^2(\Omega)$, and hence the homogenized equation can be written in non-divergence form and the diffusion matrix can be symmetrized. In other words, \eqref{eq:homog-q} can be rewritten as
\begin{equation}
  \label{eq:homeq-sym}
  \left\{
    \begin{aligned}
	  &-\overline{q}^{\rm sym}_{ij} \partial_{i} \partial_{j} u = f, \quad &&\text{in } \Omega,\\
	  &u = g \quad &&\text{on } \partial \Omega,
    \end{aligned}
  \right.
\end{equation}
where $\overline{q}^{\rm sym}$ is the symmetrization of $q$. Since $a=m\tilde a$ is precisely the symmetric part of $q$, we have
\begin{equation}
  \label{eq:qbar}
  \ol q^{\rm sym} = \int_{\bT^d} (I+\nabla_y \chi(y)) a(y) (I+\nabla_y \chi(y))^{\rm T}\, dy.
\end{equation}
 
It suffices to check $\ol a$ defined in \eqref{eq:abar} agrees with $\ol q^{\rm sym}$. To this end, we prove that the solution $\chi^j$ of \eqref{eq:celleq2}, $j=1,\dots,d$, coincides with the solution $\widetilde{\chi}^j$ of \eqref{eq:cellprob}. Multiply on both sides of \eqref{eq:cellprob} by $m$, we see that $\tilde{\chi}^j$ solves
\begin{equation*}
  -a_{ik}(y)\partial_{y_i}\partial_{y_k} \tilde{\chi}^j(y) - b_i(y)\partial_{y_i} \tilde{\chi}^j(y) = b_j(y), \qquad \text{in } \bT^d.  
\end{equation*}
Note that by elliptic regularity (we only need $\tilde a$ to be H\"older and $\tilde b$ be bounded), $\tilde{\chi}^j \in H^2(\bT^d)$, and hence we also have
\begin{equation*}
  -q_{ik}(y)\partial_{y_i}\partial_{y_k} \tilde{\chi}^j(y) - b_i(y)\partial_{y_i} \tilde{\chi}^j(y) = b_j(y), \qquad \text{in } \bT^d.  
\end{equation*}
In view of the weak formulations of the cell problems \eqref{eq:celleq2} in this section, the definition of $(q_{ij})$, the anti-symmetry of $(\phi_{ij})$, the $H^2$-regularity of the $\tilde{\chi}^j$'s again, and the relation $\partial_{y_i}q_{ij} = \partial_{y_i}a_{ij} + \beta_j = b_j$ (see \eqref{eq:betavec}), we verify that $\tilde{\chi}^j$ satisfies 
\begin{equation*}
  -\partial_{y_i}(q_{ik}(y) \partial_{y_k} \tilde{\chi}^j)  = b_j(y).
\end{equation*}
In view of the relation in \eqref{eq:betavec} again, we see that the above equation is just a rewriting of \eqref{eq:celleq2}. By the uniqueness of the solution to the cell problem, we conclude that $\widetilde{\chi}^j = \chi^j$, for all $j=1,\dots,d$. Now by comparing \eqref{eq:abar} with \eqref{eq:qbar} we conclude that $\ol a = \overline{q}^{\rm sym}$. This reproves the qualitative homogenization result in Theorem \ref{thm:homog}. 

It remains to prove $\ol a \ge \lambda_1 \mathbf{I}_d$ where $\lambda_1$ is defined in the proof of Corollary \ref{cor:q}. For the homogenization of \eqref{eq:heteq2-1} in divergence form, due to the ellipticity condition \eqref{eq:qelliptic} for $(q_{ij})$, we can apply Theorem 3.2 in Chapter 1 of \cite{MR503330} which says the homogenized matrix $(\ol q_{ij})$ satisfies $\ol q \ge \lambda_1 \mathbf{I}_d$. That is, the homogenized diffusion matrix enjoys the same ellipticity bound. Since $\ol a = \overline{q}^{\rm sym}$, $a$ has the same lower ellipticity bound.

\begin{remark}\label{rem:C11} The $C^{1,1}$ regularity of $\Omega$ was used above to show that the homogenized equation \eqref{eq:homog-q} for \eqref{eq:heteq2} can be written in non-divergence form. In numerical methods, piecewise regular domains such as polyhedral ones (in $\R^3$) are often used but they fail to be $C^{1,1}$. For convex Lipschitz domains, Smears and S\"uli \cite{MR3077903} 
established solvability in $H^2(\Omega)$ of the equation in non-divergence form (if $d\ge 3$, under the additional assumption that $\tilde a$ satisfies Cordes condition); see also \cite{MR4520025}. Theorem \ref{thm:homog} and item (1) of Theorem \ref{thm:rates} continue to hold in those settings. Other convergence rates and regularity results in Theorem \ref{thm:rates} and in Theorem \ref{thm:Lipreg} require $C^{1,\beta}$ regularity for $\Omega$, even for the homogenization of elliptic equations in divergence form.
\end{remark}
\subsection{Convergence rates}

We need to quantify the homogenization of \eqref{eq:heteq2}. Since the right hand side is $f_\eps = f(x)m(\tfrac{x}{\eps})$ which depends on $\eps$, we introduce another problem:
\begin{equation}
 \label{eq:heteq2-1}
 \left\{
 \begin{aligned}
 &-\partial_i\left(q_{ij}(\tfrac{x}{\eps})\partial_j v_\eps\right)(x) = f(x) \quad &&\text{ in } \Omega,\\
 &v_\eps(x) = g(x) \quad &&\text{ on } \partial \Omega.
 \end{aligned}
 \right.
 \end{equation} 
Then standard homogenization theory shows that $v_\eps \to u$ weakly in $H^1(\Omega)$ as $\eps \to 0$. Moreover, since the right hand side above is fixed for all $\eps$, the convergence rate is quantified by standard theory, namely, by Corollary 7.1.3 of \cite{MR3838419}. 

We also need to estimate the difference $u_\eps-v_\eps$. It satisfies the Dirichlet problem
\begin{equation}
\label{eq:uvdiff}
-\partial_i\left( q_{ik}(\tfrac{x}{\eps}) \partial_k (u_\eps -v_\eps)\right) = f(x)[m(\tfrac{x}{\eps})-1] \; \text{ in } \Omega, \quad u_\eps - v_\eps = 0 \; \text{ on } \partial \Omega.
\end{equation}
We use the trick in the proof of Lemma \ref{lem:flux} again. Since $m(y)-1$ is mean zero in $\bT^d$, there exists a unique function $h$ so that
\begin{equation}
\label{eq:uv-h}
\Delta_y h(y) = m(y) -1 \;\text{in }\,\bT^d, \qquad \int_{\bT^d} h(y)\,dy = 0.
\end{equation}
Since $m\in C^{0,\alpha}(\bT^d)$, by the standard elliptic regularity theory $h \in C^{2,\alpha}(\bT^d)$. We then have the following results.

\begin{lemma}\label{lem:fepsH-1} Assume \eqref{eq:Ab0Hypo} and \eqref{eq:bcenter}. There is a universal constant $C\in (1,\infty)$ so that for all $f\in H^1(\Omega)$
\begin{equation}
\label{eq:fepsH-1}
\|f(x)[m(\tfrac{x}{\eps})-1]\|_{H^{-1}(\Omega)} \le C\eps\,\|f\|_{H^1(\Omega)},
\end{equation}
and
\begin{equation}
\label{eq:uvepsbdd}
\|u_\eps - v_\eps\|_{H^1_0(\Omega)} \le C\eps\|f\|_{H^1(\Omega)}.
\end{equation}
\end{lemma}
\begin{proof} It suffices to prove there exists a universal constant $C < \infty$ so that
\begin{equation}
\label{eq:fH-1-key}
\left| \int_\Omega f(x) [m(\tfrac{x}{\eps})-1] \varphi(x)\, dx\right| \le C\eps \|f\|_{H^1(\Omega)}\|\nabla \varphi\|_{L^2(\Omega)}, \qquad \forall \varphi \in C^\infty_c(\Omega).
\end{equation}
Using the function $h$ defined earlier, the integral on the left hand side above can be written as
\begin{equation*}
\eps^2 \int_\Omega f(x)\varphi(x) \Delta_x\left(h(\tfrac{x}{\eps})\right)\,dx = \eps \int_\Omega f\varphi \partial_\ell[(\partial_\ell h)(\tfrac{x}{\eps})] = -\eps\int_\Omega (\partial_\ell h)(\tfrac{x}{\eps})[\varphi \partial_\ell f + f\partial_\ell \varphi].
\end{equation*}
For \eqref{eq:uv-h} we use the $C^{2,\alpha}$ elliptic regularity estimate and \eqref{eq:mbound} to get a uniform bound on $\|\nabla h\|_{L^\infty}$. Then \eqref{eq:fH-1-key} follows, and \eqref{eq:fepsH-1} is proved.

Finally, using the standard energy estimate
\begin{equation*}
  \|u_\eps -v_\eps\|_{H^1_0(\Omega)} \le C\|f[m(\tfrac{x}{\eps}) - 1]\|_{H^{-1}(\Omega)}
\end{equation*}  
for \eqref{eq:uvdiff}, we get \eqref{eq:uvepsbdd}.
\end{proof}

\begin{proof}[Proof of Theorem \ref{thm:rates}] The $L^r$, $W^{1,p}$, $W^{-1,p}$ and $H^2$ norms below are all for the domain $\Omega$. We hence omit writing $\Omega$ explicitly.

\emph{Proof of (1)}: In view of the uniform ellipticity and the regularity of $q$ in \eqref{eq:qelliptic}, and by Corollary 7.1.3 of \cite{MR3838419}, for $f\in L^2$ and $g\in H^2$, we get
\begin{equation*}
 \|v_\eps - u\|_{L^2} + \|v_\eps - u - \{\Phi_{\eps,j}(x)-x_j\}\partial_j u\|_{H^1_0} \le C\eps \|u\|_{H^2} \le C\eps \left(\|f\|_{L^2} + \|g\|_{H^{2}}\right)
\end{equation*} 
for some universal constant $C < \infty$. Under the further assumption that $f\in H^1(\Omega)$, we can combine the estimate above with \eqref{eq:uvepsbdd} to obtain \eqref{eq:H1rate} and \eqref{eq:L2rate} of Theorem \ref{thm:rates}.

\emph{Proof of (2)}: Now $g=0$ and $p\in (1,d)$. Let $1/r = 1/p-1/d$, let $v_\eps$ solve \eqref{eq:heteq2-1} with $g=0$. Apply Theorem 7.5.1 of \cite{MR3838419} to this equation, we first get  $\|v_\eps - u\|_{L^r} \le C_r \eps\|f\|_{L^p}$ for some $C_r$ that only depends on the data in \eqref{eq:Ab0Hypo} and on $r$. The difference $u_\eps - v_\eps$ is still characterized by \eqref{eq:uvdiff}. Apply the uniform $W^{1,p}$ regularity result in Theorem 5.3.1 of \cite{MR3838419} associated to the operator $-\partial_i(q_{ij}(x/\eps)\partial_j)$ (note that $q$ is H\"older and hence VMO), for some constant $C_p$ depending on the data in \eqref{eq:Ab0Hypo} and on $p$, we have
\begin{equation}
\label{eq:uvdiffW}
\|u_\eps -v_\eps\|_{W^{1,p}} \le C_p\,\|f(x)[m(\tfrac{x}{\eps})-1]\|_{W^{-1,p}} \le C\eps\,\|f\|_{W^{1,p}}.
\end{equation}
The last inequality is obtained by repeating the argument in Lemma \ref{lem:fepsH-1}. By Sobolev embedding, 
the above still holds if the left hand side is changed to $\|u_\eps - v_\eps\|_{L^r}$. Combine all the results above we obtain \eqref{eq:Lqrate}. 

\emph{Proof of (3)}: The proof is almost the same as above. Apply Theorem 7.5.1 of \cite{MR3838419} to the problem \eqref{eq:heteq2-1} with $g=0$. We get $\|v_\eps - u\|_{L^\infty} \le C\eps \|f\|_{L^p}$. To estimate the difference $u_\eps - v_\eps$, we note that \eqref{eq:uvdiffW} holds, and because $p>d$, the inequality is still true if the left hand side is replaced by $\|u_\eps -v_\eps\|_{L^\infty}$.
Combine the results above we get \eqref{eq:Linfrate} and finish the proof of Theorem \ref{thm:rates}.
\end{proof}

\subsection{Uniform Lipschitz estimates}

We prove Theorem \ref{thm:Lipreg} as a direct consequence of the uniform Lipschitz regularity theory provided in Chapter 5 of \cite{MR3838419}. Applying Theorem 5.6.2 there to the equation \eqref{eq:heteq2}, we can find a constant $C>0$ depending only on the data in \eqref{eq:Ab0Hypo} and on $p,\eta$ so that
\begin{equation*}
\|\nabla u_\eps\|_{L^\infty(\Omega)} \le C\{\|g\|_{C^{1,\eta}(\partial \Omega)} + \|f_\eps\|_{L^p(\Omega)}\},
\end{equation*}
provided that $f_\eps \in L^p$ and $p>d$. Since $f\in L^p$ for some $p>d$, $m \in L^\infty(\bT^d)$ and $f_\eps(x) = f(x)m(\frac{x}{\eps})$, we have $\|f_\eps\|_{L^p} \le C\|f\|_{L^p}$ for some universal constant $C\in (0,\infty)$. We hence get \eqref{eq:Lipreg} and finish the proof of Theorem \ref{thm:Lipreg}.

\section{One dimensional examples}
\label{sec:1D}

In this section we study the one-dimensional setting and make several comments on the centering condition for the drift.

\subsection{The centering condition in laminated media} As explained in Remark \ref{rem:1D}, for laminated media, the study of the invariant measures reduces to that in the one dimensional setting. 
We hence consider the following equation on the torus $\bT := \R /\Z$ (and assume $\tilde a, \tilde b$ are smooth):
\begin{equation}
\label{eq:invm1D}
(\tilde a(y)m(y))'' - (\tilde b(y)m(y))' = 0, \qquad y \in \bT.
\end{equation}
Here the prime denotes derivative in $y$. 
Let $\tilde m(y) := \tilde a(y)m(y)$ and rewrite the equation above as
\begin{equation*}
\tilde m''(y) - \left(\frac{\tilde b(y)}{\tilde a(y)} \tilde m(y)\right)' = 0.
\end{equation*}
Integrate this equation once, we get
\begin{equation}
\label{eq:1D-1}
\tilde m'(y) - \tilde b(y)m(y) = C_1.
\end{equation}
Integrate again, we get
\begin{equation*}
m(y) = \frac{1}{\tilde a(y)} \left\{C_0 \exp\left(\int_0^y \frac{\tilde b(s)}{\tilde a(s)}\,ds\right) + C_1 \int_0^y \exp\left(\int_{y'}^y \frac{\tilde b(s)}{\tilde a(s)} \,ds\right) dy'\right\}.
\end{equation*}
The constants $C_0$ and $C_1$ are determined by the periodicity of $\tilde m$ and by the fact that $\int_\bT m =1$.
In view of \eqref{eq:1D-1}, the \emph{centering} condition is equivalent to $C_1 = 0$, and it holds if and only if
\begin{equation*}
\int_{\bT^1} \frac{\tilde b(y)}{\tilde a(y)}\,dy = 0.
\end{equation*}
The invariant measure is then given by
\begin{equation*}
m(y) =\frac{1}{\tilde a(y)} \exp\left(\int_0^y \frac{\tilde b(s)}{\tilde a(s)}\,ds \right)\Big\slash \int_0^1 \frac{1}{\tilde a(y)}\exp\left(\int_0^y \frac{\tilde b(s)}{\tilde a(s)}\,ds \right)\,dy.
\end{equation*}
The above also verifies that, for laminated media in higher dimensions, the centering condition is precisely \eqref{eq:1Dcenter}.

\subsection{Some comments}

To check the sharpness of Theorem \ref{thm:Lipreg} in terms of the order of regularity, we consider the 1D equation 
\begin{equation}
\label{eq:1D-2}
-\tilde a(\tfrac{x}{\eps})\left(u_\eps\right)'' - \tfrac1\eps \tilde b(\tfrac{x}{\eps}) \left(u_\eps\right)' = 0 \;\,\text{ in }\, (0,1),
\end{equation}
where $\tilde b$ is of the form $\tilde b(y) = \tilde a(y)b(y)$. In view of \eqref{eq:1Dcenter}, the centering condition is $\int_{\bT} b = 0$. For example, set $b(y) =\cos(2\pi y)$; we compute and get
\begin{equation*}
m(y) = \frac{\tilde m(y)}{\tilde a(y)}, \qquad \tilde m(y) = C_0\exp\left(\frac{\sin(2\pi y)}{2\pi}\right)
\end{equation*}
with some normalization constant $C_0 > 0$. Multiply on both sides of \eqref{eq:1D-2} by $m(\frac{x}{\eps})$. Then we get
\begin{equation*}
-\left(\tilde m(\tfrac{x}{\eps}) u'_\eps(x)\right)' = 0.
\end{equation*}
The unique solution $u_\eps$ with boundary data $u_\eps(0)=0$ and $u_\eps(1) = 1$ then satisfies
\begin{equation*}
u'_\eps(x) = \frac{c_\eps}{\tilde  m(\tfrac{x}{\eps})}, \quad \frac{1}{c_\eps} = \int_0^1 \tilde m^{-1}(\tfrac{x}{\eps})\,dx.
\end{equation*}

This simple one dimensional example shows that, for $\tilde b \ne 0$ under the centering condition, one cannot expect to have uniform in $\eps$ regularity that is higher (smoother) than Lipschitz in general. Compare this with the case of $\tilde b = 0$. The simple 1D equation at the beginning with boundary condition $u_\eps(0) =0$ and $u_\eps(1) = 1$ then has smooth solution.

\medskip

Next, we comment on the necessity of the centering condition. It is known already in \cite{MR503330} that when the centering condition fails one cannot expect to have a homogenization result like Theorem \ref{thm:homog}. As a simple example in 1D, consider the problem 
\begin{equation}
\left(u_\varepsilon\right)''+\tfrac{1}{\varepsilon}\left(u_\varepsilon\right)' = f \quad \text{in } (0,1), \quad \text{and} \quad 
u_\varepsilon(0)=u_\varepsilon(1)=0,
\end{equation}
and $f(x) \equiv 1$. It is clear that the invariant measure is $m(y) \equiv 1$ on $\bT$, and the periodic drift vector $b(y) \equiv 1$ fails the centering condition. Direct computation then shows 
$$
u_\varepsilon(x) = \eps\,\left(x- \frac{1-e^{-x/\eps}}{1-e^{-1/\eps}}\right).
$$
It follows that $u_\eps \to u$ uniformly in $[0,1]$ where $u(x) \equiv 0$. Clearly $u$ cannot be a solution to a uniformly elliptic equation with right hand side $f \equiv 1$.  

\section{Concluding remarks}
\label{sec:conclusions}

In this paper, we studied quantitative homogenization of uniformly elliptic equations with a periodic diffusion matrix and a large drift term. We show that when the drift satisfies the centering condition \eqref{eq:bcenter}, the equation can be transformed to divergence form without any drift. We can then transfer almost all of the recently developed sharp quantitative estimates, including convergence rates in various norms and uniform Lipschitz regularity results, to the setting of this paper. We also comment on the necessity of the centering condition and on the sharpness of the results.

Our method is quite flexible. For example, one may consider the more general equation
\begin{equation*}
-\tilde a_{ij}(\tfrac{x}{\eps}) \partial_i\partial_{j} u_\eps - \tfrac{1}{\eps} \tilde b_j(\tfrac{x}{\eps}) \partial_j u_\eps - \tilde c_j(\tfrac{x}{\eps})\partial_j u_\eps + \tfrac{1}{\eps} q_1(\tfrac{x}{\eps}) u_\eps + q_0(\tfrac{x}{\eps}) u_\eps = f,
\end{equation*}
say, for $u_\eps \in H^1_0(\Omega)$. Qualitative theory for the above equation without the large potential was treated already in \cite{MR503330}. The large potential case in divergence form without the drifts was considered by Zhang \cite{MR4260009}; a combination of the technique there with our method can be used for the problem above. As long as $O(\eps^{-2})$ potential is considered, see e.g.\,Allaire and Orive \cite{MR2351401}, the large potential will affect the homogenization at the highest order and more involved transformation or analysis will be needed. We leave it to future studies. Note also, since the key transformation in this paper does not involve the boundary conditions, we expect that our method continues to work for Neumann boundary problems of \eqref{eq:heteq} (for nonzero data, this amounts to oscillatory Neumann data). Our method is hence a clear improvement of the classical approaches which are restricted, as remarked in \cite{MR503330}, to homogeneous Dirichlet boundary conditions.

\section*{Acknowledgments}
The authors are grateful to Hung V.\,Tran for helpful discussions on properties of invariant measures. They would like also to thank the anonymous referees for their valuable suggestions which helped the authors improve the quality of the paper.

\bibliographystyle{plain}{}
\bibliography{/Users/wenjiajing/RefAllTime}
%
%

\end{document}